\documentclass[12pt]{amsart}
\usepackage[margin=1in]{geometry}
\usepackage[english]{babel}
\usepackage[utf8]{inputenc}
\usepackage{subcaption}
\usepackage{amsmath}
\usepackage{amssymb}
\usepackage{amsfonts}
\usepackage{amsthm}
\usepackage{mathrsfs}
\usepackage[all]{xy}
\usepackage{graphicx}
\usepackage{color}
\usepackage{cite}
\usepackage{url}
\usepackage[labelfont=bf,labelsep=period,justification=raggedright]{caption}
\usepackage[english]{babel}
\usepackage[utf8]{inputenc}
\usepackage{hyperref}
\usepackage[colorinlistoftodos]{todonotes}
\usepackage{tkz-fct}
\usepackage{tikz}
\usepackage{tabularx}
\usetikzlibrary{calc}

\renewcommand{\qedsymbol}{$\blacksquare$}

\theoremstyle{theorem}
\newtheorem{theorem}{Theorem}

\theoremstyle{definition}
\newtheorem*{definition}{Definition}

\newtheorem{conjecture}[theorem]{Conjecture}
\newtheorem{corollary}[theorem]{Corollary}
\newtheorem{proposition}[theorem]{Proposition}
\newtheorem{lemma}[theorem]{Lemma}
\newtheorem{example}{Example}

\newtheorem{algorithm}{Algorithm}

\theoremstyle{plain}

\begin{document}

\title{Diophantine \emph{m}-tuples of Triangular Numbers}
\author{Sounak Bagchi, Christian Zhou-Zheng}
\maketitle
\begin{abstract}
    
A $m$-tuple with the property $D(n)$ is a tuple of $m$ positive integers $(a_1, a_2, \dots, a_m)$ such that  $a_i a_j + n$ is an square, for $1 \le i < j \le m$. The $k$th triangular number is $T_k = \frac{k(k+1)}{2}$ for nonnegative integers $k$. We consider $D(1)$ tuples consisting only of triangular numbers. We prove the nonexistence of any $D(1)$ triangular quadruple and describe an algorithm to generate an infinite family of $D(1)$ triangular triples, which we conjecture contains all $D(1)$ triangular triples.
\\

We then consider general $D(n)$ tuples. To aid with computational difficulties, we present an efficient algorithm, using Generalized Pell Equations (GPEs), to determine whether $T_a$ is in a $D(n)$ triangular pair, which runs in $O(a^{1/2})$ time. We then prove that no $D(n)$ triangular pair exists for $n \equiv 2,5 \text{ (mod } 9\text{)}$, and discuss other values of $n$ for which there appear to be no $D(n)$ triangular pairs. We also show that our $D(n)$ equation has solutions in all $\mathbb{Q}_p$, for $p \neq 3$. We then present progress on determining a general criteria on $n$ for which no $D(n)$ triangular pairs exist.
\end{abstract}

\section{Introduction} \label{prelims}
\begin{definition}
    A \textbf{$D(n)$ Diophantine $m$-tuple of rationals} is a set of $m$ positive rationals such that the product of any two distinct elements of the set is $n$ less than a rational perfect square. If all rationals in the tuple are integers, then we say it is a $D(n)$ Diophantine $m$-tuple of integers.
\end{definition}

The question of what numbers form a $D(1)$ $m$-tuple dates back to Diophantus, who found a rational $D(1)$ $4$-tuple, or quadruple, given by $$\left\{\frac{1}{16}, \frac{33}{16}, \frac{17}{4}, \frac{105}{16} \right\}.$$ Pierre de Fermat was first to find a $D(1)$ integer quadruple, given by
$$\{1, 3, 8, 120 \}.$$
It was Euler \cite{b4} who found a rational extension of this quadruple to the set $$\left\{1, 3, 8, 120, \frac{777840}{8288641} \right\}.$$

The case of Diophantine $m$-tuples for $D(1)$ is of particular interest to many mathematicians studying this topic. For example, Dujella proved in \cite{b5} that every $D(1)$ integer quadruple could be extended to a $D(1)$ rational quintuple. More recently, a longstanding conjecture on the nonexistence of an integer $D(1)$ Diophantine quintuple was proven in \cite{b6}:

\begin{theorem}[He, Togb\'{e}, Ziegler]\label{no quintuple}
    There does not exist an integer $D(1)$ Diophantine quintuple.
\end{theorem}

Theorem \ref{no quintuple} will be useful in our next setction. In this paper, we examine Diophantine m-tuples through a class of numbers:

\begin{definition}
    The sum of the first $n$ positive integers is known as the $n$\textbf{-th triangular number}, given by $T_n = \frac{n(n+1)}{2}$.
\end{definition} 

Throughout this paper, we will only consider Diophantine integer $m$-tuples consisting solely of triangular numbers, which we refer to as \emph{triangular $D(n)$ tuples.} We focus primarily on the cases $D(n)$ for positive $n$, though our results can be extended naturally to negative $n$. Immediately, Fermat's integer quadruple yields a $D(1)$ triangular triple:

\begin{example}\label{eg:tfermat}
    The subtriple $\{1, 3, 120 \}$ of the integer $D(1)$ quadruple $\{1, 3, 8, 120 \}$ is a $D(1)$ triple consisting solely of triangular numbers, since $1 = T_1, 3 = T_2,$ and $120 = T_{15}$. We can check that $1 \cdot 3 + 1 = 2^2$, $1 \cdot 120 + 1 = 11^2$, and $3 \cdot 120 + 1 = 19^2$.
\end{example}

Going forward, we will omit the term ``Diophantine" when describing $D(n)$ Diophantine $m$-tuples. We also refer to $\mathbb{Z} / m \mathbb{Z}$ as the ring of integers modulo $m$, and $\mathbb{Z}_p$ as the $p$-adic integers.

\section{$D(1)$ Triangular Tuples} \label{$D(1)$}

$D(1)$ was historically the first condition to be studied, and has the greatest coverage in the literature; it is the natural first condition to study in the context of triangular $D(n)$ tuples. In particular, classifying Diophantine $m$-tuples for different values of $m$ will be of interest to us. We first present two trivial proofs:

\begin{proposition}\label{prop:infpair}
    Infinitely many triangular $D(1)$ pairs exist. 
\end{proposition}
\begin{proof}
    The tuple $\{T_n,T_{n+4}\}$ is a $D(1)$ pair for all $n$: \begin{align*}
        T_n\cdot T_{n+4}+1&=\frac{n(n+1)}{2}\cdot\frac{(n+4)(n+5)}{2}+1\\&=\frac{n^4+10n^3+29n^2+20n+4}{4}\\&=\left(\frac{n^2+5n+2}{2}\right)^2.\qedhere
    \end{align*}
\end{proof}

\begin{proposition}
    There exists no triangular $D(1)$ quadruple.
\end{proposition}
\begin{proof}
    Suppose, for the sake of contradiction, that there existed a triangular $D(1)$ quadruple $\{T_a,T_b,T_c,T_d\}$. Notice that \[T_a\cdot 8+1=\frac{a(a+1)}{2}\cdot 8+1=4a^2+4a+1=(2a+1)^2\] (resp. $T_b,T_c,T_d$). 8 is not a triangular number and therefore is not equal to $T_a, T_b, T_c, T_d$, so $\{T_a,T_b,T_c,T_d,8\}$ would be an integer $D(1)$ quintuple, contradicting Theorem \ref{no quintuple}.
\end{proof}

Therefore we have at least an infinite family of pairs and knowledge that no quadruples exist. It then naturally follows to ask what results we may find on triangular $D(1)$ triples. Indeed, there exists an infinite family of triples that is closely related to the family of pairs.

\begin{proposition}\label{prop:inftrip}
    There exist infinitely many triangular $D(1)$ triples.
\end{proposition}
\begin{proof}
    The tuple $\{T_n,T_{n+4},T_{4n^2+20n+8}\}$ is a $D(1)$ triple for all $n$. We saw this for $\{T_{n}, T_{n+4} \}$ in Proposition~\ref{prop:infpair}; now, \begin{align*}
        &T_n\cdot T_{4n^2+20n+8}+1\\&=\frac{n(n+1)}{2}\cdot\frac{(4n^2+20n+8)(4n^2+20n+9)}{2}+1\\&=4n^6+44n^5+157n^4+202n^3+103n^2+18n+1\\&=(2n^3+11n^2+9n+1)^2.\\\\&T_{n+4}\cdot T_{4n^2+20n+8}+1\\&=\frac{(n+4)(n+5)}{2}\cdot\frac{(4n^2+20n+8)(4n^2+20n+9)}{2}+1\\&=4n^6+76n^5+557n^4+1938n^3+3123n^2+1862n+361\\&=(2n^3+19n^2+49n+19)^2.\hspace{11.6em}\square
    \end{align*}
    \renewcommand{\qedsymbol}{}
\end{proof}

Both Proposition \ref{prop:infpair} and Proposition \ref{prop:inftrip} are results of the following theorem, first discovered by Euler:

\begin{theorem}
    Suppose that $ab + 1 = r^2$ for integer $a,b$. Then $\{a, b , a + b + 2r, 4r(r+a)(r+b) \}$ is a $D(1)$ quadruple.
\end{theorem}

Particularly, by taking $a = 8$ and $b = T_n$, it follows that $\{a, b, a + b + 2r \} = \{8, T_n, T_{n+4} \}$ is a $D(1)$ triple, and $\{a, b, a + b + 2r, 4r(r+a)(r+b) \} = \{8, T_n, T_{n+4}, T_{4n^2 + 20n + 8} \}$ is a $D(1)$ quadruple.\footnote{In fact, a slightly different form of Euler's discovery for Triangular Numbers is used in \cite{b6.5}, where one may arrive at the results presented by using $A = 8$ and $B = T/2$ in the main theorem of the paper.}

Consider that by letting $\{T_a,T_b,T_c\}=\{T_n,T_{n+4},T_0\}$ (in which case $T_0=0$), we obtain the above triple as $\{T_a,T_b,T_{4ab+2a+2b-c}\}$. 
In fact, we can show the following:

\begin{theorem}\label{proof of alg}
Let $S = \{ \{T_{n}, T_{n+4}, T_{4n^2 + 20n + 8} \}: n \in \mathbb{N} \} \cup \{ \{T_1, T_2, T_{15} \} \}.$ Define the function $L$ on a triple of triangular numbers to be $$L( \{T_a, T_b, T_c \}) = \{T_a, T_b, T_{4ab + 2a + 2b - c} \}.$$ 
If $\{T_a, T_b, T_c \} \in S$ is a $D(1)$ triple with $a \le b \le c$, then $L(\{T_a, T_b, T_c \})$, $L(\{T_a, T_c, T_b \})$, and $L(\{T_b, T_c, T_a \})$ are also $D(1)$ triples. 
\end{theorem}
\begin{proof}
    Consider an arbitrary set $s = \{T_a, T_b, T_c \} \in S$ that is a $D(1)$ triple, and set $$x = \sqrt{T_a T_b + 1},$$ $$y = \sqrt{T_a T_c + 1},$$ $$z = \sqrt{T_b T_c + 1}.$$ Since $s$ is a $D(1)$ triple, it follows that $x,y,z$ are positive integers. Now, consider the set $L(s) = \{T_a, T_b, T_{4ab + 2a + 2b - c} \}$ generated. Let 
    \begin{align*}
    x_1 &= \sqrt{T_a T_b + 1},\\
    y_1 &= \sqrt{T_a T_{4ab + 2a + 2b - c} + 1}, \\
    z_1 &= \sqrt{T_b T_{4ab + 2a + 2b - c} + 1}.
    \end{align*}
    Clearly $x_1 = x$ is an integer. The key observation is that the following equations hold: $$y_1 = \frac{x^2 - T_a (T_a - 8)}{y}$$ $$z_1 = \frac{x \cdot y_1 - 2a - 1}{T_a}.$$ 
    
    The above formulations imply that $y_1$ and $z_1$ are rational, since all components involved in their expressions are rational. Moreover, $y_1^2 = T_a T_{4ab + 2a + 2b - c} + 1$ and $z_1^2 = T_b T_{4ab + 2a + 2b - c} + 1$ are clearly integral. So, it follows that $y_1$ and $z_1$ are integral. Hence, $x_1, y_1, z_1$ are all integral, and $\{T_a, T_b, T_{4ab + 2a + 2b - c} \}$ is a $D(1)$ triple. The same formulas work upon the reorderings $\{T_a, T_c, T_b \}$ and $\{T_b, T_c, T_a \}$ of $s$, with the same rearrangement on the indices. (In other words, for the set $\{T_a, T_c, T_b \}$, instead of $x = \sqrt{T_a T_b + 1 }$ we set $x = \sqrt{T_a T_c + 1}$, and so forth throughout the proof.)
\end{proof}

 Note that $\{T_n, T_{n+4}, T_0 \}$ is not a $D(1)$ triple since $T_0 = 0$ is not nonzero. However, for illustratory purposes, we forgo the nonzero restriction in this section and consider them $D(1)$ triples, since they possess the same properties of other $D(1)$ triples. The above proof does not apply for tuples of the form $\{T_0, T_b, T_c \}$ since we get a $0$ in the denominator of $z_1$, but the proof that the generating step works is trivial anyways for tuples containing $T_0$.
 \\
 
 Theorem \ref{proof of alg} tells us that all elements of $S$ are $D(1)$ triples, since $S$ starts off with a basis of $D(1)$ triples, and every addition to the set is subsequently also a $D(1)$ triple. 

 \begin{example}
    Consider the set $s = \{T_2, T_6, T_{64} \}$, which is a $D(1)$ triple of the form \\ $\{T_n, T_{n+4}, T_{4n^2 + 20n + 8} \}$ for $n=2$. Note that $s \in S$.
    
    Consider a rearrangement of $s$, the set $s' = \{T_2, T_{64}, T_6 \}$. As per our definitions in the proof of Theorem \ref{proof of alg}, let $$x = \sqrt{T_2 T_{64} + 1} = 79,$$ $$y = \sqrt{T_2 T_6 + 1} = 8,$$ $$z = \sqrt{T_{64} T_6 + 1} = 209.$$ Observe that $4(2)(64) + 2(2) + 2(64) - 6 = 638$, so $$L(s') = \{T_2, T_{64}, T_{638} \}.$$ Now, $$x_1 = \sqrt{T_2 T_{64} + 1} = 79,$$ $$y_1 = \sqrt{T_2 T_{638} + 1} = 782,$$ $$z_1 = \sqrt{T_{64} T_{638} + 1} = 20591.$$ Hence, $\{T_2, T_{64}, T_{638} \}$ is a $D(1)$ triple, and is also an element of $S$. Moreover, one can check that the formulas hold: clearly $x_1 = x$, and $$y_1 = \frac{79^2 - T_2(T_2 - 8)}{8} = 782,$$ $$z_1 = \frac{79 \cdot 782 - 4 - 1}{T_2} = 20591.$$
\end{example}
 
\begin{example}\label{eg:genalg}
    Let $\to$ denote an example construction of a triangular $D(1)$ triple through the $L$ function, as defined in Theorem \ref{proof of alg}. Then \[\{T_1,T_2,T_{15}\}\to\{T_2,T_{15},T_{153}\}\to\{T_{15},T_{153},T_{9514}\},\] and it is seen that \begin{align*}T_{15}T_{153}+1&=1189^2,\\T_{15}T_{9514}+1&=73699^2,\\T_{153}T_{9514}+1&=730234^2.\end{align*} So, $\{T_{15}, T_{153}, T_{9514} \}$ is a $D(1)$ triangular triple.
\end{example}

Observe that each element of the infinite family of triples specified in Proposition~\ref{prop:inftrip}, along with the set $\{T_1, T_2, T_{15} \}$, generates an infinite family of its own, constructed by repeatedly applying the procedure in Theorem \ref{proof of alg}. It is possible, however, that not all triples are generated through this algorithm. For example, we included the ``special" triple $\{T_1,T_2,T_{15}\}$, which doesn't fall under the $\{T_n, T_{n+4}, T_{4n^2+20n+8} \}$ parametrization, in our set $S$. What if other triples like $\{T_1, T_2, T_{15} \}$ exist? It is not obvious that all triangular $D(1)$ triples are obtained from the algorithm.
\\

However, deeper analysis reveals that the triple $\{T_1, T_2, T_{15} \}$ is not as sporadic as it may seem. First, we expand our definition of triangular numbers to include negative indices:

\begin{lemma}
    $T_{-n}=T_{n-1}$.
\end{lemma}
\begin{proof}
    $T_{-n}=\frac{-n(-n+1)}{2}=\frac{n(n-1)}{2}=T_{n-1}$.
\end{proof}

In particular, note $T_0=T_{-1}$. Now in fact we see we may apply the procedure of Theorem \ref{proof of alg} to yield $\{T_1,T_2,T_{15}\}$ from a triple of the form $\{T_n,T_{n+4},T_{4n^2 + 20n + 8}\}$: let $n=1$, then with $\to$ denoting an application of Theorem \ref{proof of alg}, we see \begin{align*}\{T_1, T_5, T_{32} \} &\to \{T_1,T_5,T_0\}=\{T_1,T_{-6},T_{-1}\}\\ &\to\{T_{1},T_{-1},T_{2} \} \to\{T_1,T_2,T_{15}\}.\end{align*}

We have studied the patterns of negatively indexed triangular numbers significantly less than those of positively indexed ones, primarily as the former are not as intuitive (what is the notion of summing the first $-n$ integers, and why is it positive?), but there appears to be further theory in this direction, which we are hopeful yields results in future work.

\section{General $D(n)$ Triangular Pairs} \label{$D(2)$}

In Section \ref{$D(1)$} we proved that infinitely many triangular $D(1)$ triples exist; it then stands to examine $D(n)$ triangular tuples for $n$ being an arbitrary integer. In particular, we will take a look at the simplest case---triangular $D(n)$ pairs. The question we would like to answer is:
$$\text{For what } n \text{ do there exist triangular } D(n) \text{ pairs? }$$ 
To find triangular pairs, we must find positive integer solutions $(a,b,c)$ to the equation 
$$\frac{a(a+1)}{2} \cdot \frac{b(b+1)}{2} + n = c^2.$$
Then, such a positive integer solution would correspond with the $D(n)$ pair $(T_a, T_b)$, with $T_a T_b + n = c^2$.

\begin{theorem}
    Fix a positive integer $a$ and an integer $n$. Let $A$ be the set of all positive integer solutions $(b,c)$ to \begin{equation}\label{$D(2)$ first}
         \frac{a(a+1)}{2} \cdot \frac{b(b+1)}{2} + n = c^2,
    \end{equation} 
    and let $B$ be the set of all positive integer solutions $(x,y)$ to 
    \begin{equation} \label{$D(2)$ second}
    x^2 - a(a+1) y^2 = 16n - a(a+1),
    \end{equation}
    where $x \equiv 0 \pmod{4}$ and $y \equiv 1 \pmod{2}$. There exists a bijective map between $A$ and $B$; in fact, such a map is described by the function $f: A \to B$ with $f((n,k)) = (4k, 2n+1)$. 
\end{theorem}
\begin{proof}
    We will rewrite Equation (\ref{$D(2)$ first}) as Equation (\ref{$D(2)$ second}), and show that the map holds and is bijective. We get 
    \begin{align*}
    \frac{a(a+1)}{2} \cdot \frac{b(b+1)}{2} + n = c^2 &\iff 4a(a+1) \cdot b(b+1) + 16n = 16c^2 \\
    &\iff a(a+1) ((2b+1)^2 - 1) + 16n = 16c^2 \\
    &\iff 16n - a(a+1) = 16c^2 - a(a+1)(2b+1)^2 \\
    &\iff (4c)^2 - a(a+1)(2b+1)^2 = 16n - a(a+1).
    \end{align*}
    The final equation reveals that solution $(b,c)$ to Equation (\ref{$D(2)$ first}) corresponds to a solution $(4c, 2b+1)$ to Equation (\ref{$D(2)$ second}), as given by the function $f$. Moreover, each step is easily reversible, hence this function $f$ is also bijective, as desired.
\end{proof}

Equation (\ref{$D(2)$ second}) describes a family of \textbf{Generalized Pell Equations} (GPEs) of the form $$x^2 - Dy^2 = N$$ for positive integers $D$ and integers $N$. This new form is much more helpful when determining the existence of solutions. The main use for Equation (\ref{$D(2)$ second}) is because of the following theorem:

\begin{theorem}\label{Bounding Theorem}
    Consider the ring $\mathbb{Z}[\sqrt{d}]$ for a positive integer $d$, and define the norm of an element $z = a + b \sqrt{d} \in \mathbb{Z}[\sqrt{d}]$ to be $N(z) = a^2 - db^2$. Let $u = a + b\sqrt{d}$ be an element of $\mathbb{Z}[\sqrt{d}]$ with $N(u) = 1$ and $a, b > 0$. Then, there exists a solution $(x,y) \in \mathbb{Z}^2$ to $x^2 - dy^2 = n$, where $n$ is a nonzero integer, if and only if there exists a solution $(x_1, y_1)$ with 
    \begin{equation}\label{Bounded}
    |x_1| \le \frac{\sqrt{\left|n\right|} (1 + \sqrt{u})}{2}, |y_1| \le \frac{\sqrt{\left|n\right|} (1 + \sqrt{u})}{2 \sqrt{d}}.
    \end{equation}
\end{theorem}
\begin{proof}
    See Theorem 2.1 from \cite{b7}.
\end{proof}

Theorem \ref{Bounding Theorem} lets us determine whether a triangular number $T_a$ is part of any Diophantine $D(n)$ pair by checking through a finite space, given by the pair of inequalities in (\ref{Bounded}).
\\

The main problem with Theorem \ref{Bounding Theorem} is determining an element $u$ in $\mathbb{Z}[\sqrt{d}]$ with norm $1$. For general $d$, this is not trivial - in fact, it is equivalent to the solving the standard Pell Equation $x^2 - dy^2 = 1$. In our case, it is much simpler.

\begin{lemma}\label{unit}
    The element $u = 2t + 1 + 2 \sqrt{t(t+1)}$ has norm $1$ in $\mathbb{Z}[\sqrt{t(t+1)}]$.
\end{lemma}
\begin{proof}
    Check: 
    \begin{align*}
        N(u) &= (2t+1)^2 - 2^2 t(t+1) \\
        &= 4t^2 + 4t + 1 - 4t(t+1) = 1. \qedhere
    \end{align*}
\end{proof}

Now, we can apply Theorem \ref{Bounding Theorem} to our GPE in Equation (\ref{$D(2)$ second}). In particular, an integer solution to this equation exists if and only if an integer solution $(x_1, y_1)$ exists with $$|x_1| \le \frac{\sqrt{|16n - a(a+1)|} \left(1 + \sqrt{2a + 1 + 2\sqrt{a(a+1)}} \right)}{2},$$  $$|y_1| \le \frac{\sqrt{|16n - a(a+1)|} \left(1 +\sqrt{2a + 1 + 2\sqrt{a(a+1)}} \right)}{2 \sqrt{a(a+1)}}.$$ Keep $n$ fixed. Asymptotically, the bound on $x_1$ is $O(a^{3/2})$ (with $O$ meaning asymptotically), while the bound on $y_1$ is $O(a^{1/2})$. So, if we were to check whether or not $T_a$ is in a $D(n)$ pair, we would:
\\

\begin{itemize}
    \item Set $u = 2a + 1 + 2 \sqrt{a(a+1)}$
    \item Check if $16n - a(a+1) + a(a+1) y^2$ is a perfect square, for all integers $y$ less than or equal to the bound in (\ref{Bounded}) \\
\end{itemize}

This algorithm runs in $O(a^{1/2})$ time, since $a,u$ are explicitly given. If we were to determine whether $T_a$ is part of a $D(n)$ pair for all $a < A$ for some integer $A$, then iterating over all $a$ leads to an algorithm that runs in $O(A^{3/2})$ time. 
\\

Why is this so useful? The naive brute-force algorithm involves checking whether $(T_a, T_b)$ is a $D(n)$ pair for fixed $n$, where $1 \le a \le c_1$ and $1 \le b \le c_2$ for some integers $c_1, c_2$. Nevertheless, the naive algorithm does not allow us to determine whether a triangular number is part of \textbf{any} possible triangular $D(n)$ pair. This modified algorithm, however, does allow us to. The following example is a good illustration.

\begin{example}
    Consider $n = 2$, and suppose we would like to check whether $T_a$ is in a $D(2)$ Triangular Pair for all $1 \le a \le 20$. We will check for $a = 20$, which induces the GPE $$x^2 - 420y^2 = 32 - 420 = -388.$$ The unit $u$ from Lemma \ref{unit} is $u = 41 + 2 \sqrt{420}$. Then, according to Theorem \ref{Bounding Theorem}, we need to check the region $$|y_1| \le \frac{\sqrt{388} \left(1 + \sqrt{41 + 2 \sqrt{420}} \right)}{2 \sqrt{420}} \approx 4.832.$$ In particular, for $y_1 = 1, 2, 3, 4$, we can check that $420y_1^2 - 388$ is not a perfect square, so no solutions $(x,y)$ exist for $a = 20$. Repeating this process for $1 \le a \le 19$ gives the algorithm as described above.
\end{example}

The algorithm not only significantly reduces the complexity of the problem, but also provides lower bounds for what triangular numbers can be part of a $D(n)$ tuple. Given a sufficiently large lower bound, the algorithm also provides reasonable suspicion as to answering ``For which integer $n$ do there not exist any $D(n)$ triangular pairs?"
\\

We implemented the algorithm\footnote{The implementation can be found at https://github.com/SounakB1/D-n-Generalized-Pell-Equation-Calculator.git.} 
 described above in SageMath, an open-source math system that builds upon Python libraries. We ran tests for triangular number $D(n)$ pairs on $1 \le n \le 50$, testing whether $T_a$ was part of a triangular number pair for $1 \le a \le 10^5$. Thirty-six such $n$ produced a $D(n)$ pair in the range.
 
 \begin{example}
     Suppose that $n = 7$. Using the bounds for $t=1$ reveals no solutions to the GPE $x^2 - 2y^2 = 110$, but for $t = 2$, the equation $x^2 - 6y^2 = 106$ has the solution $(x,y) = (16,5)$. From the bijective map, this corresponds to the solution $(T_{t}, T_{\frac{y-1}{2}}) = (T_2, T_2)$, which is indeed a $D(7)$ triangular pair as $3 \cdot 3 + 7 = 4^2$.
 \end{example}
 
 For fourteen such $n$, no integer solutions $(x,y)$ in the range were found. They are: $$n = 2, 5, 11, 12, 14, 17, 20, 23, 29, 32, 38, 41, 42, 47$$ Why is this the case? For some $n$, this is explainable; others, not as much.

It appears there do not exist any $D(9k+2)$ or $D(9k+5)$ triangular pairs, based on data from $1 \le n \le 50$, for nonnegative integer $k$.  We'll prove this claim in this section. We'll also revert back to using Equation (\ref{$D(2)$ first})---now that we are trying to disprove the existence of solutions, we no longer need the GPE form. 
\\

The proof is by contradiction. We start off with a simple observation. Suppose that $$T_a T_b + 3k + 2 = c^2,$$ for positive integers $a,b,c$ and integer $k$. Note that $3 \mid T_a$ if and only if $a \equiv 0, 2 \pmod{3}$. Thus, if either $a \equiv 0, 2 \pmod{3}$ or $b \equiv 0, 2 \pmod{3}$, then modulo $3$ the equation becomes $$2 \equiv c^2 \pmod{3},$$ which is impossible. Hence, we may assume that $a, b \equiv 1 \pmod{3}$.

\begin{theorem}
Let $n$ be an integer congruent to $2$ or $5$ modulo $9$. Then, no triples of positive integers $(a,b,c)$ exist with $$T_a T_b + n = c^2.$$ In other words, no $D(9k + 2)$ or $D(9k + 5)$ pairs of triangular numbers exist, for all integers $k$.
\end{theorem}
\begin{proof}
Assume for the sake of contradiction that such a triple $(a,b,c)$ exists. We must have $a, b \equiv 1 \pmod{3}$, since $n \equiv 2 \pmod{3}$. Then, observe that $$T_a \equiv T_b \equiv \frac{1(1+1)}{2} \equiv 1 \pmod{3}.$$ Hence, $T_a T_b \equiv 1 \pmod{3}$. This implies that, modulo $3$, $$c^2 = T_a T_b + n \equiv 1 + 2 \equiv 0 \pmod{3}.$$ From this, $3 \mid c^2$, so $3 \mid c$ and $9 \mid c^2$. Hence, $$T_a T_b + n \equiv 0 \pmod{9}.$$ Consider the value of $a$ modulo $9$. If $a \equiv 1 \pmod{9}$, then $$T_a \equiv \frac{1(1+1)}{2} \equiv 1 \pmod{9}.$$ If $a \equiv 4 \pmod{9}$, then $$T_a \equiv \frac{4(4+1)}{2} \equiv 10 \equiv 1 \pmod{9}.$$ Finally, if $a \equiv 7 \pmod{9}$, then $$T_a \equiv \frac{7(7+1)}{2} \equiv 28 \equiv 1 \pmod{9}.$$ Hence, because $a \equiv 1 \pmod{3}$, it follows that $T_a \equiv 1 \pmod{9}$. Similarly, $T_b \equiv 1 \pmod{9}$, so $$T_a T_b + n \equiv 1 + n \equiv 3, 6 \pmod{9},$$ a contradiction since neither $3$ nor $6$ are quadratic residues.
\end{proof}
\subsection{Sporadic Cases}
The sporadic cases of $n = 12, 17, 42$, for which there appear to be no $D(n)$ triangular doubles, do not seem to follow any pattern. It is possible that the search space in which we are checking, i.e. $1 \le a \le 10^5$, is not large enough, and that for each $n$ in this list, there exist $x, y > 10^5$ where $T_x T_y + n$ is a perfect square. However, it seems unlikely that this is the case, as the range $1 < y < 5$ produces solutions for other $n$. In fact, we propose the following:

\begin{conjecture}\label{conj}
    There are infinitely many integers $n \not\equiv 2, 5 \pmod{9}$ such that no $D(n)$ triangular number pairs exist. In other words, infinitely many sporadic $n$ exist for which no $D(n)$ triangular pairs exist.
\end{conjecture}

One thing we can prove is that such $n$ do not arise as a result of a congruential obstruction. In other words, we will prove the following:

\begin{theorem}\label{big theorem n}
Fix a positive integer $n$, and let $m$ be a positive integer with $9 \nmid m$. Then, there exist positive integers $a,b,c$ such that $$T_a T_b + n = c^2 \pmod{m}.$$ In other words, the equation has solutions in $\mathbb{Z}/m \mathbb{Z}$. \end{theorem}

We'll take this by parts. Like most proofs of this nature, our general idea is to:

\begin{itemize}
    \item Show that the equation has solutions in $\mathbb{Z}/p\mathbb{Z}$
    \item Extend the proof to $\mathbb{Z}/p^k\mathbb{Z}$ for integer $k$ using Hensel's Lemma
\end{itemize}

Let's tackle the first part. We've already showed that this is true in $\mathbb{Z} / 3 \mathbb{Z}$. we'll prove the following result:

\begin{theorem}\label{sum of squares}
    For fixed integers $a,b,c$, there exists a solution to the equation $ax^2 + by^2 \equiv c \pmod{p}$ for all primes $p$ such that $p \nmid a,b$. 
\end{theorem}
\begin{proof}
    Consider all values of the form $$\frac{c-by^2}{a} \pmod{p},$$ where $a,b,c$ are fixed. Note that $y^2$ takes on $\frac{p+1}{2}$ values (the number of quadratic residues modulo $p$). Hence, the expression $\frac{c-by^2}{a}$ modulo $p$ also takes on $\frac{p+1}{2}$ values, since $a,b,c$ are constant and not divisible by $p$. But note that the equation $ax^2 + by^2 \equiv c \pmod{p}$ is equivalent to $$x^2 \equiv \frac{c-by^2}{a} \pmod{p}.$$ Since $x^2$ takes on $\frac{p+1}{2}$ values modulo $p$, and $\frac{p+1}{2} + \frac{p+1}{2} = p+1$, by the Pigeonhole Principle, there exists a solution $(x,y)$ to the congruence, as desired.
\end{proof}

This can be extended to the GPE form of our triangular $D(1)$ equation, i.e. in Equation (\ref{$D(2)$ first}).

\begin{corollary}\label{squarefree}
    The equation $x^2 - t(t+1)y^2 = 16n - t(t+1)$ has a solution $(x,y,t)$ modulo $p$, for any prime $p$ and any integer $n$. As a result, the equation $$T_a T_b + n = c^2$$ has a solution $(a,b,c)$ modulo $p$, for any integer $n$.
\end{corollary}
\begin{proof}
    For $p = 2$, the solution $(0,0,0)$ works modulo $2$. For $p \neq 2$, choose any $t$ for which $p \nmid t(t+1)$. Then, Theorem \ref{sum of squares} leads to the result. 
\end{proof}
We are now done with the first step of the process. The result of Corollary \ref{squarefree} is enough to show that our equation for $D(n)$ triangular pairs has solutions in any $\mathbb{Z} / m \mathbb{Z}$ where $m$ is squarefree, using the Chinese Remainder Theorem. Now, we proceed with the second by attempting to "lift" our solutions to higher prime powers, using a multivariate form of Hensel's Lemma \cite{b8}.

\begin{theorem}\label{weak version}
    The equation $$T_a T_b + n = c^2$$ has solutions in $\mathbb{Z}/p^k\mathbb{Z}$, where $p > 3$, $n$ is fixed and $a,b,c$ are variable.
\end{theorem}

\begin{proof}
    Consider the GPE form, $$x^2 - t(t+1) y^2 = 16n - t(t+1).$$ Consider the function $$f(t,x,y) = x^2 - t(t+1) y^2 + t(t+1) - 16n.$$ We want to set this equation to $0$ modulo powers of $p$. Using the Multivariable Form of Hensel's Lemma, since we already have solutions modulo $p$ for all primes $p$, we look at the partial derivatives for all $3$ variables: $$(\nabla f)(t,x,y) = ((2t+1)(1 - y^2), 2x, 2t(t+1) y).$$ For $p > 3$, we can choose $t = 1$ and $x,y$ accordingly (using Theorem \ref{sum of squares}) so that none of the components are $0$. Hence, using the Multivariable Hensel's Lemma, we arrive at our conclusion.
\end{proof}

Note that, due to the Chinese Remainder Theorem, Theorem \ref{weak version} extends to all $\mathbb{Z} / m \mathbb{Z}$ with $2, 3 \nmid m$.

\begin{theorem}
    The equation $T_a T_b + n = c^2$ has solutions $(a,b,c)$ in $\mathbb{Z} / 2^k \mathbb{Z}$, with $k \ge 1$, for all values of $n$.
\end{theorem}
\begin{proof}
    We prove a smaller claim first.
\\

    \textbf{Claim:} For $0 \le k \le 2^{m}-1$, $T_k$ modulo $2^m$ takes on every value in the set $\{0, 1, \dots, 2^m - 1\}$. 
    \\

    \emph{Proof:} Suppose that this is a contradiction, and that two values of $a,b$ exist for which $T_a \equiv T_b \pmod{2^m}$ and $a \neq b$. Then, it follows that $$\frac{a(a+1)}{2} - \frac{b(b+1)}{2} \equiv \frac{(a-b)(a+b+1)}{2} \equiv 0 \pmod{2^m}.$$ Note that $a-b$ and $a+b+1$ have opposite parity, so either $a \equiv b \pmod{2^{m+1}}$ or $a + b \equiv 2^{m+1} - 1 \pmod{2^{m+1}}$. The latter cannot be true since the maximum sum possible is $2 \left(2^{m} - 1 \right) = 2^{m+1} - 2$, hence $a \equiv b \pmod{2^m}$. So, $T_k$ has a unique residue modulo $2^m$ for $0 \le k \le 2^{m} - 1$, hence implying the conclusion. \quad\quad\quad\quad\quad\quad\quad\quad\quad\quad\quad\quad\quad\quad\quad\quad\quad $\square$
    \\ 
   
    Our claim makes the conclusion obvious, as $T_a T_b$ can take on any value. For example, set $T_a \equiv 4-n \pmod{2^k}$ and $T_b \equiv 1 \pmod{2^k}$, so that $T_a T_b + n \equiv (4-n) + n \equiv 2^2 \pmod{2^k}$. 
\end{proof}

This finally tells us gives our proof of Theorem \ref{big theorem n}, by using the Chinese Remainder Theorem on all $\mathbb{Z} / p^k \mathbb{Z}$ with $p \neq 3$ and $\mathbb{Z} / 3 \mathbb{Z}$. 
\\

Since we know solutions always exist modulo powers of a prime, apart from $p=3$, it is curious that we do not find any solutions for some $n \neq 2, 5 \pmod{9}$. It is certainly not impossible, but may provide further insight into this problem.

\subsection{Towards a full characterization}

Viewing this problem from a GPE perspective, lots of theory surrounding GPEs has been developed, often centered around \textbf{continued fractions}. Nevertheless a general criteria for telling whether a GPE has any integer solutions has not yet been found. Theorems and algorithms exist, however, for characterizing the structure of the set of solutions to a GPE. Of particular interest is a theorem from \cite{b10}.

\begin{theorem}\label{convergent}
If positive integers $x,y$ satisfy $x^2 - dy^2 = n$ with $|n| < \sqrt{d}$, then $x/y$ is a convergent of the continued fraction expansion of $\sqrt{d}$.   
\end{theorem}

\begin{proof}
    See K. Conrad's proof in \cite{b11}.
\end{proof}

Our generalized Pell Equation, of the form $$x^2 - t(t+1) y^2 = 16n - t(t+1),$$ does not always fit the parameter $|n| < \sqrt{d}$ in Theorem \ref{convergent}. However, we can make some conclusions that aid with this.
\\

\begin{theorem}\label{cf of triangular num}
The continued fraction expansion of $\sqrt{t(t+1)}$ is $$\sqrt{t(t+1)} = t + \frac{1}{2 + \frac{1}{2t + \frac{1}{2 + \frac{1}{2t +\dots}}}} = [t; 2, 2t, 2, 2t, \dots].$$
\end{theorem}

\begin{proof}
    Let $a =t +\frac{1}{2 + \frac{1}{2t + \frac{1}{2 + \frac{1}{2t +\dots}}}}$. Then, observe that $$a = t + \frac{1}{2 + \frac{1}{a+t}} = t + \frac{1}{\frac{2a+2t+1}{a+t}} = t + \frac{a+t}{2a+2t+1} = \frac{2at + 2t^2 + a + 2t}{2a+2t+1}.$$ This can be rearranged as $$2a^2 = 2t^2 + 2t \iff a^2 = t^2 + t \iff a = \sqrt{t(t+1)},$$ as desired. 
\end{proof}

Observe the first few convergents of the continued fraction expansion of $\sqrt{t(t+1)}$: $$\frac{t}{1}, \frac{2t+1}{2}, \frac{4t^2 + 3t}{4t+1}, \dots$$ Now consider the equation $x^2 - t(t+1)y^2 = N$ for $|N| \le t$. From our bound on $|y_1|$ in Theorem \ref{Bounding Theorem}, note that, to have any solution to $x^2 - t(t+1)y^2 = N$, we must have a solution $(x_1, y_1)$ with

\begin{align*} |y_1| &\le \frac{\sqrt{|N|}(1 + \sqrt{2t+1 + 2\sqrt{t(t+1)}})}{2 \sqrt{t(t+1)}} \\
&\le \frac{\sqrt{t}\left(1 + \sqrt{2t+1 + 2\sqrt{t(t+1)}}\right)}{2 \sqrt{t(t+1)}} \\
&= \frac{1 + \sqrt{2t+1 + 2\sqrt{t(t+1)}}}{2 \sqrt{t+1}} \\
&= \frac{1 + \sqrt{2t+1 + 2\sqrt{t(t+1)}}}{\sqrt{4t + 4}} \\
&< \frac{1 + \sqrt{2t+1 + 2t + 1}}{\sqrt{4t+4}} = \frac{1 + \sqrt{4t+2}}{\sqrt{4t+4}}. 
\end{align*}

For $t \ge 1$, this implies that $|y_1| < 2$, hence $y_1 = 0, 1$. (We don't consider $-1$ as it provides no new solution.) But we also know that $\frac{x_1}{y_1}$ is a continued fraction convergent of $\sqrt{t(t+1)}$; thus we must have $$x_1 = t, y_1 = 1 \implies x^2 - t(t+1)y^2 = t^2 - t(t+1) = -t.$$ $y_1 =0$ is a special case, which yields that $N$ must be a perfect square. These two observations lead to the following theorem: 

\begin{theorem}\label{existence less than sqrt}
    Consider the equation $x^2 - t(t+1)y^2 = N$ for positive integer $t$ and integer $N$ satisfying $|N| \le t$. Then, an integer solution $(x,y)$ exists if and only if $N = -t$ or $N = a^2$ for some integer $a$.
\end{theorem}

As mentioned, this is not so useful for our purposes yet, since our GPE is of the form $$x^2 - t(t+1)y^2 = 16n - t(t+1),$$ and $16n - t(t+1) \not\le t.$ However, \textbf{Lagrange's system of reductions}, as described in \cite{b9}, allows us to reduce our GPE. we describe the algorithm below.

\begin{algorithm}
    Suppose we have $x^2 - Dy^2 = N$ for $N^2 > D$. Note that the condition $k^2 \equiv D \pmod{|N|}$ is necessary, for some value of $k$, as $$x^2 - Dy^2 \equiv 0 \pmod{|N|} \iff D \equiv \left(\frac{y}{x} \right)^2 \pmod{|N|}.$$ Moreover, such a $k$ must exist for which $0 \le k \le |N|/2$. For all $k$ in this range satisfying $k^2 \equiv D \pmod{|N|}$, one can reduce the equation $x^2 - Dy^2 = N$ to the equation $$x^2 - Dy^2 = \frac{k^2 - N}{D} = h$$ Then, the process is repeated again until $h^2 < D$. So, in this way, we get multiple "branches" of equations from a single equation $x^2 - Dy^2 = N$. If one of these branches has a solution, the original solution to the equation can be found by back substitution (though we need not worry about this aspect, for we are only testing for existence). If none of the branches have a solution, the original equation $x^2 - Dy^2 = N$. does not either.
\end{algorithm}

\begin{example}
    Consider the GPE $x^2 - 2y^2 = 94$, which is our characteristic GPE with $t = 1$ and $n = 6$. Note that $$40^2 \equiv 2 \pmod{94},$$ so one possible "branch" that can be reached is from taking the new GPE $$x^2 - 2y^2 = \frac{40^2 - 2}{94} = 17.$$ Once again, $6^2 \equiv 2 \pmod{17}$ (this being the solution between $0$ and $17/2$), so our equation becomes $$x^2 - 2y^2 = \frac{6^2 - 2}{17} = 2.$$ One more reduction can be performed using $0^2 \equiv 2 \pmod{2}$, to get $$x^2 - 2y^2 = \frac{0^2 - 2}{2} = -1.$$ This has a solution from Theorem \ref{existence less than sqrt}, as $N = -t = -1$. In this way, we see that there exists a $D(6)$ triangular pair.
\end{example}

In this example, we see how useful it is to have Theorem \ref{existence less than sqrt}, for we can stop immediately once we reach an equation $x^2 - t(t+1)y^2 = N$ with $|N| \le t$. It's only necessary for us to perform Lagrange's System of Reductions.
\\

Making this into a general criteria or algorithm of existence for all values of $t$ (assuming we fix $D(n)$) is, of course, much more difficult. Two key steps are involved: determining whether $k^2 \equiv D \pmod{|N|}$ has a solution, and finding all such solutions. In the context of the first step of our GPE, this means determining whether $$k^2 \equiv t(t+1) \pmod{|16n - t(t+1)|}$$ has a solution, and finding it. For some sufficiently large value of $t$ it follows that $t(t+1) > 16n$, so we can rewrite this as $$k^2 \equiv t(t+1) \pmod{t(t+1) - 16n} \iff k^2 \equiv -16n \pmod{t(t+1) - 16n}.$$ So, the problem, in part, comes down to determining for what values of $t$ is $-16n$ a quadratic residue modulo $t(t+1) - 16n$, for fixed $n$.  
\\

The discussion is rather hard to continue without choosing a fixed $n$; suppose $n = 12$, the first nontrivial $n$ for which we did not find $D(n)$ triangular pairs. Then, we must determine for which $t$ is $-192$ a quadratic residue modulo $t(t+1) - 192)$. One way to do this using the Chinese Remainder Theorem. For that, we need the following well-known theorem:
\\

\begin{theorem}\label{powers of 2}
    The equation $x^2 \equiv a \pmod{2^j}$ has solutions if and only if $a = 0$ or $a = 4^k (8b+1)$, for some nonnegative integers $b$ and $k$.
\end{theorem}

Suppose now that $t(t+1) - 192 = 2^{e} p_1^{e_1} \dots p_i^{e_i}$ for odd primes $p_1, \dots, p_k$. It then follows that $k^2 \equiv -192 \pmod{t(t+1) - 192}$ has a solution if and only if each of the congruences 
\begin{align*}
k_{0}^2 &\equiv -192 \pmod{2^e}, \\
k_{1}^2 &\equiv -192 \pmod{p_1^{e_1}}, \\
&\text{   } \vdots \\
k_{i}^2 &\equiv -192 \pmod{p_i^{e_i}}
\end{align*}
have solutions. In fact, from Hensel's Lemma, it suffices to check if the set of congruences 
\begin{align*}
k_{0}^2 &\equiv -192 \pmod{2^e}, \\
k_{1}^2 &\equiv -192 \pmod{p_1}, \\
&\text{   } \vdots \\
k_{i}^2 &\equiv -192 \pmod{p_i}
\end{align*}
has solutions. The first condition we can handle with Theorem \ref{powers of 2}. For the remaining congruences, note that $$k_j^2 \equiv -192 \pmod{p_i} \iff \left(\frac{k_j}{8} \right)^2 \equiv -3 \pmod{p_i},$$ so it really suffices to check whether $-3$ is quadratic residue modulo $p_j$ for $1 \le j \le i$. But this is quite simple; in fact, it is just a result of quadratic reciprocity: $$\left(\frac{-3}{p} \right) = \begin{cases} 
      1 & \text{ if } p \equiv 1, 7 \pmod{12} \\
      -1 & \text{ if } p \equiv 5, 11 \pmod{12}
   \end{cases}
$$

Thus, it follows that $t(t+1) - 192$ can only have prime factors congruent to $1, 7$ modulo $12$. 
\\

One should also not forget about the edge case $p = 3$; in fact, $$k^2 \equiv -3 \pmod{3^j}$$ has no solutions for $j > 1$. So, it follows that $192 - t(t+1)$ cannot be divisible by $9$.
\\

When combining these facts using Theorem \ref{powers of 2}, this gives a necessary and sufficient condition for determining if $-192$ is a quadratic residue modulo $192 - t(t+1)$, given the prime factorization of $t(t+1) - 192$.
\\

Finding a general prime factorization, or determining what primes divide $t(t+1) - 192$, is a another difficult task. Determining the general decomposition of a quadratic in $\mathbb{F}_p$ is a task currently being looked at as an avenue. We do believe that this may produce promising results in the future (not only for $D(12)$, but for general $D(n)$). 

\section{Acknowledgments}
We'd like to thank our research mentor, Dr. Simon Rubinstein-Salzedo, for introducing us to this topic, guiding us throughout the research process, holding weekly meetings for discussion, and reviewing drafts of our paper. We'd also like to thank Evan Chen for providing feedback on the final draft of this paper. Finally, we'd like to thank the Euler Circle community for being extremely supportive and providing help whenever necessary.

\end{document}